\documentclass[11pt,reqno]{amsart}
\usepackage{amscd,graphicx,psfrag,amsxtra}
\usepackage{pinlabel}
\usepackage{hyperref}
\usepackage{xcolor}
\newcommand{\mtilde}{\tilde M}
\newcommand{\ntilde}{\tilde N}
\newcommand{\ltilde}{\tilde L}

\newcommand{\Z}{\mathbb Z}

\newcommand{\LL}{\mathcal L}

\newcommand{\Sm}{\mathcal S}

\newcommand{\R}{\mathbb R}

\renewcommand{\phi}{\varphi}

\newcommand{\lk}{\operatorname{lk}}

\newcommand{\lktilde}{\widetilde{lk}}
\newcommand{\slktilde}{\widetilde{slk}}

\newcommand{\sign}{\operatorname{sign}}

\DeclareMathOperator{\TOP}{{TOP}}
\newcommand{\rp}{\R P}

\newcommand{\la}{\langle}
\newcommand{\ra}{\rangle}

\newcommand{\conn}{\mathbin{\#}} 

\newcommand{\spinc}{\ifmmode{\operatorname{Spin}^c}\else{$\operatorname{spin}^c$\ }\fi}

\newtheorem{theorem}{Theorem}[section]

\newtheorem{lemma}[theorem]{Lemma}
\newtheorem{proposition}[theorem]{Proposition}

\newtheorem{thm}{Theorem}

\theoremstyle{definition}
\newtheorem{definition}[theorem]{Definition}
\newtheorem{remark}[theorem]{Remark}

\title{Topological spines of 4-manifolds}
\thanks{The second author was partially supported by NSF Grant DMS-1811111.}
\author[Hee Jung Kim]{Hee Jung Kim}
\address{Department of Mathematics, MS 9063\newline\indent Western Washington University\newline\indent Bellingham, WA 98225}
\email{\rm{heejungorama@gmail.com}}
\author[Daniel Ruberman]{Daniel Ruberman}
\address{Department of Mathematics, MS 050\newline\indent Brandeis
University \newline\indent Waltham, MA 02454}
\email{\rm{ruberman@brandeis.edu}}

\begin{document}
\begin{abstract}
We show that infinitely many of the simply connected $4$-manifolds constructed by Levine and Lidman that do not admit PL spines actually admit topological spines. 
\end{abstract}
\maketitle
\section{Introduction}
A remarkable recent paper of Levine and Lidman~\cite{levine-lidman:spineless} gives examples $X$ of smooth (equivalently, PL) simply connected $4$-manifolds, homotopy equivalent to a $2$-sphere, that do not admit PL spines. In other words, there is no PL (not necessarily locally flat) embedded sphere $S$ in $X$ that is a strong deformation retract of $X$. (See~\cite{cappell-shaneson:spineless,matsumoto:spineless,matsumoto:wild} for results in higher dimensions and about spines for non-simply connected manifolds). In this note we show that an infinite family of the Levine-Lidman examples (those whose intersection form is $\langle 4 \rangle$) admit a (tame) {\em topological spine}; this is a {\em locally} PL sphere that is a strong deformation retract of $X$.   

Finding a spine for a compact manifold (especially in codimension $2$) is typically phrased as a problem of homology surgery~\cite{cappell-shaneson:surgery}, as one is looking for a suitable complement for a regular neighborhood of the spine. The complement would be a homology cobordism between the boundary of this regular neighborhood and the boundary of $X$.  Indeed, the argument of Levine and Lidman is to find a $4$-manifold $X$, homotopy equivalent to a $2$-sphere, such that $\partial X$ is not smoothly homology cobordant to integral surgery on any knot in the $3$ sphere. In contrast, we show that for a family of their examples, $\partial X$ is {\em topologically} homology cobordant to integral surgery on a knot in the $3$ sphere. With some additional information about the fundamental group of the homology cobordism, this is sufficient to build the spine.

To state our main theorem, we use the notation of~\cite{levine-lidman:spineless} where $\{W_p\mid p \in \Z\}$ denotes a family of oriented smooth $4$-manifolds, each homotopy equivalent to a sphere, where a generator of $H_2(W_p)$ has self-intersection $4$.  For $p \not \in \{-2, -1, 0\}$,  the manifold $W_p$ has no PL spine~\cite[Proposition 3.3]{levine-lidman:spineless}.
\begin{thm}\label{thm:spine}
For each $p$, the manifold $W_p$ contains a tame topological spine.
\end{thm}

The paper~\cite{levine-lidman:spineless} also includes examples of spineless $4$-manifolds whose intersection forms are $\langle \pm k^2 \rangle$, for any non-zero $k$. It is natural to ask if these manifolds have topological spines; at present our method does not seem to extend to values of $k$ greater than $2$. We will comment on this further at the end of the paper.  We remark that a recent paper of Hayden and Piccirillo~\cite{hayden-piccirillo:spines} contains examples of smooth/PL manifolds homotopy equivalent to a $2$-sphere of arbitrary self-intersection with no PL spine but which contain a topological spine. 
\subsection*{Acknowledgments}
Thanks to Adam Levine and Tye Lidman for interesting correspondence, and to Patrick Orson for a useful discussion on some fine points of surgery theory.  The second author was partially supported by NSF Grant DMS-1811111.

\section{Preliminaries}
First, we give the definition of a topological spine, following~\cite{pedersen:spines}.
\begin{definition}\label{D:spine}
A spine of a compact topological manifold $W$ with non-empty boundary is a tamely embedded complex $S \subset W$ so that $S$ is a strong deformation retract of $W$ and $S \hookrightarrow W$ is a simple homotopy equivalence. A PL spine in a PL manifold $W$ is a spine that is a subcomplex.
\end{definition}
In this paper, `spine' will always mean the topological version, unless specified that it is to be PL. We will be looking for a rather simple sort of spine in a manifold $W$ homotopy equivalent to a $2$-sphere, whose intersection form is $\langle n \rangle$, for some integer $n$. The complex $S$ will be a $2$-sphere that is locally flat, except near a single point, where it is a cone on a knot $K$. Such an $S$ has a neighborhood $N(S)$ homeomorphic to $X_4(K)$, where $X_n(K)$ denotes the $4$-manifold obtained by attaching an $n$-framed $2$-handle to the $4$-ball along $K$. We use the standard notation $S^3_n(K)$ for $\partial X_n(K)$.

We briefly recall the construction of Levine and Lidman. For any integer $m$, let $Q_m$ denote the orientable (as a manifold) circle bundle over $\rp^2$ with twisted Euler class $m$. We assume throughout that $m = -4p-3$ for an integer $p$; note that $m$ being odd implies that $H_1(Q_m) \cong \Z_4$.  The fundamental group $G_m$ of $Q_m$ has order $4\cdot |m|$ and in fact is a semi-direct product
\begin{equation}\label{E:Gm}
1 \to \Z_{|m|} \to G_m  \to \Z_4 \to 1
\end{equation}
with a generator of $\Z_4$ acting on the subgroup $ \Z_{|m|}$ by multiplication by $-1$.

A number of surgery diagrams for $Q_m$ are given in~\cite{levine-lidman:spineless} to which we add (the well-known) variation below; the box represents $p$ full twists.

\begin{figure}[htb]
\labellist
\small\hair 2pt
\pinlabel {$p$} [ ] at 50 190
\pinlabel {$-1$} [ ] at 20 330
\pinlabel {$0$} [ ] at -18 255
\pinlabel {$K$} [ ] at 133 255
\pinlabel {$J$} [ ] at 137 80
\endlabellist
\centering
\includegraphics[scale=0.5]{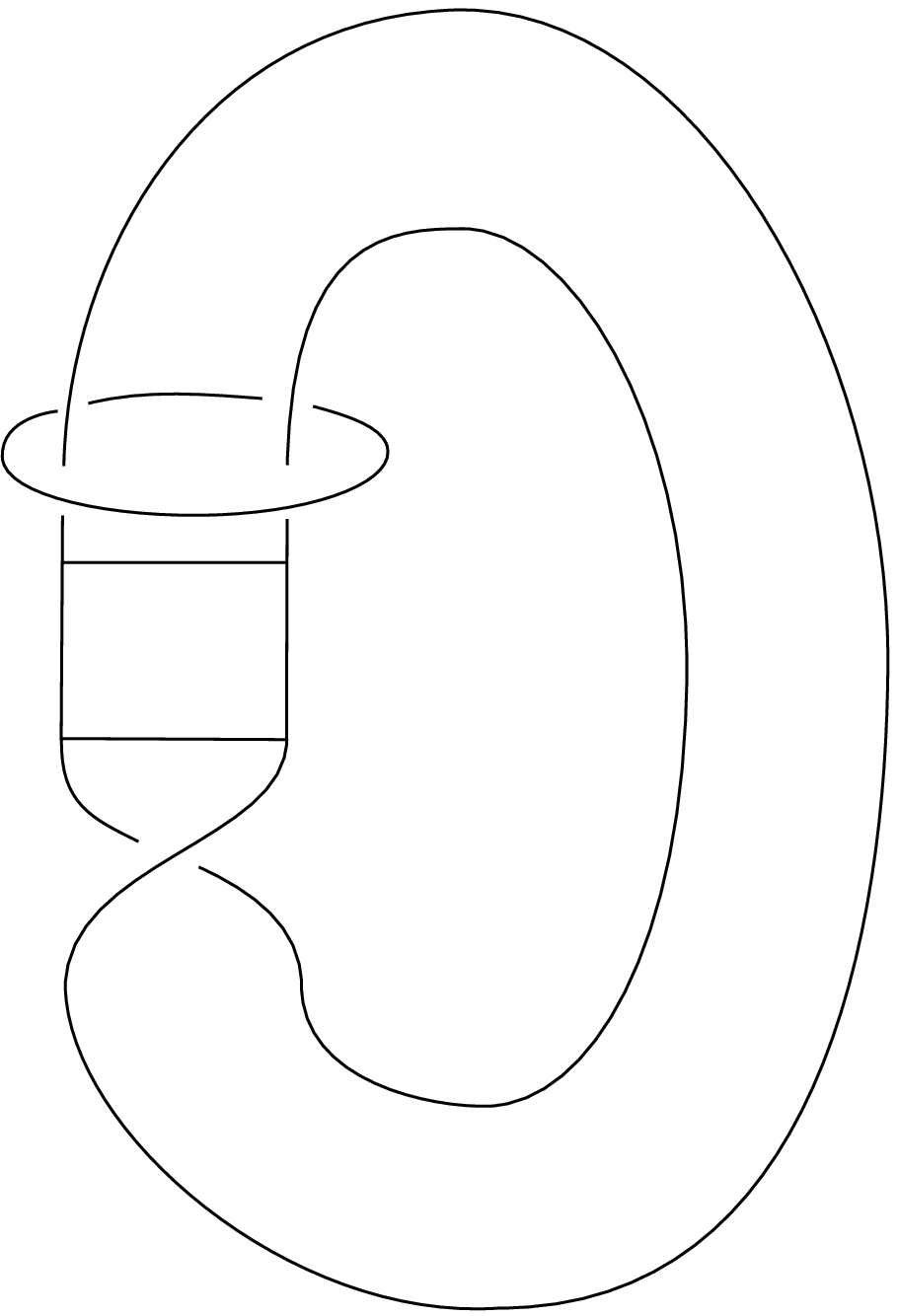}
\caption{Circle bundle $Q_{-4p-3}$}
\label{F:Qm}
\end{figure}
Note that for $p=0, -1$ the component labeled $J$ can be blown down to yield that $Q_m$ is surgery on a knot in the $3$-sphere. (It is not clear what happens for $p=-2$.) More generally, $Q_m$ is $+4$ framed surgery on a knot in $Y_p$ drawn in~\cite[Figure 2]{levine-lidman:spineless}. 
The manifold $W_p$ of Levine-Lidman has boundary $M_p = Q_{m} \conn  -Y_p$, where $Y_p$ is a certain Brieskorn homology sphere. The homology sphere summand is important to the construction of $W_p$ but will not play much of a role for us. 
\begin{remark}
We follow the convention of~\cite{levine-lidman:spineless} by indexing some objects by $p$ and others by $m$. 
\end{remark}

\subsection{Twisted linking numbers}\label{s:twisted}
Suppose that $M$ is an oriented $3$-manifold, and that $p: \mtilde \to M$ is a regular covering space with covering group $\pi$. Suppose further that $H_1(\mtilde) =0$. Suppose that $L_1$ and $L_2$ are null-homotopic knots in $M$. Then we can consider linking numbers $\lk_{\mtilde}$ between lifts of $L_1$ and $L_2$. Fixing a particular lift $\ltilde_1$ and $\ltilde_2$ of each component, we can form a generating function
$$
\lktilde(L_1,L_2) = \sum_{g \in \pi} \lk_{\mtilde}(\ltilde_1,\ltilde_2\cdot g)g.
$$
Note that by invariance of linking numbers under (orientation-preserving) homeomorphisms, the function $\lktilde$ determines linking numbers between arbitrary lifts of $L_1$ and $L_2$.  More precisely, $\lk_{\mtilde}(\ltilde_1\cdot h,\ltilde_2\cdot g) = \lk_{\mtilde}(\ltilde_1,\ltilde_2\cdot gh^{-1})$.

If $(L,\phi)$ is a framed knot in $M$, then we define the self-linking $\slktilde(L;\phi)$ by choosing $L_1 = L$ and $L_2$ to be the push-off $L^\phi$ defined by the framing. The constant term in this generating function is the framing of $\ltilde$. More precisely, the choice of lift $\ltilde$ determines a lift $\ltilde^\phi$, which is the push-off of $\ltilde$ by the lifted framing. Since $L$ and $\ltilde$ are null-homologous, the framings can be characterized by integers $n$ and $n'$.  Suppose that the covering group $\pi$ is finite, of order $d$, and consider a chain $\tilde{C}$ in $\mtilde$ with boundary $\ltilde$, projecting to a null-homology $C$ for $L$. Then each intersection point of $L^\phi$ with $C$ lifts to $d$ intersection points in $\mtilde$. Taking signs into account, we see that 
\begin{align}\label{E:lift} 
d n & = \sum_{g,h \in \pi}\lk_{\mtilde}(\ltilde \cdot g, \ltilde^\phi \cdot h)\notag\\
& = d \left( \sum_{g \in \pi}\lk_{\mtilde}(\ltilde, \ltilde^\phi \cdot g)\right) \ \text{and so}\notag\\
n & = n' +  \left( \sum_{1\neq g \in \pi}\lk_{\mtilde}(\ltilde, \ltilde^\phi \cdot g)\right) =  n' +  \left( \sum_{1\neq g \in \pi}\lk_{\mtilde}(\ltilde, \ltilde \cdot g)\right).
\end{align}

Such linking numbers have been considered by many authors~\cite{cha-ko:covering-links,powell:twisted, przytycki-yasuhara:linking} mostly in the setting when $\pi$ is a cyclic group, where there are good formulas (originating in Seifert's work on branched covers~\cite{seifert:covers}) in terms of Seifert matrices. For more general covering groups $\pi$, we can also compute $\lktilde(L_1,L_2)$ in terms of data in $M$. This procedure seems to be well-known~\cite{goldsmith:linking} so we will just summarize it.

Suppose that the covering $\mtilde \to M$ corresponds to a homomorphism $r: \pi= \pi_1(M,x_0) \to \pi$ where $x_0$ is a basepoint. Note that the choice of lifts $\ltilde_i$ is equivalent to choosing a homotopy class $\beta_i$ of arcs (modulo multiplication by elements of $\ker(r)$) from $x_0 \in M$ to each $L_i$. Write $c_i$ for the endpoint of $\beta_i$ on $L_i$.  For each $i$, choose an immersion $\delta_i: D^2 \to M$ whose restriction to the boundary is $L_i$.  Then to any intersection point $x$ of $L_1$ and $\Delta_2 = \delta_2(D^2)$ we assign an element $g(x) \in \pi$ as follows.  Note that $\Delta_2$ is based, using $\beta_2$. Then we define
$$
g(x) = \beta_1 \ast \alpha_1 \ast \alpha_2^{-1} \ast \beta_2^{-1}
$$
where $\alpha_1$ is a path in $L_1$ from $c_1$ to $x$, and $\alpha_2$ is the projection of a path in $D^2$ to a path in $\Delta_2$ from $x$ to $c_2$.

Then~\cite[Theorem 4.10]{goldsmith:linking} we have
\begin{equation}\label{E:linking}
\lktilde(L_1,L_2) = \sum_{x \in L_1 \cap \Delta_2} \sign_x(L_1, \Delta_2)\, r(g(x))
\end{equation}
where $\sign_x(L_1, \Delta_2)$ is the local intersection number at $x$.

From a minor variation on Equation \eqref{E:linking} we get the following recipe for varying the twisted linking number.
\begin{lemma}\label{L:clasp}
Let $(L_1,\; L_2)$ be a link with both components null-homotopic. Then for any element $q = \sum a_g g \in \Z[\pi]$, there is a knot $L_1(q)$ isotopic to $L_1$ such that $\lktilde(L_1(q),L_2) = \lktilde(L_1,L_2) + q$.
\end{lemma}
\begin{proof}
It clearly suffices to prove this for $q = \pm g$; the procedure is shown in Figure~\ref{F:clasp}. In summary, send out a feeler from $L_1$ to $L_2$ following along the loop $g$ (and the base paths for $L_1$ and $L_2$) and put a clasp whose sign is determined by the sign in $q$.
\end{proof}
If $L$ is a knot, then we can insert a self-clasping of either sign along a loop $g\neq 1$ to get a new knot $L(g)$ such that the linking number $\lk_{\mtilde}(\ltilde(g),\ltilde(g) \cdot g)$ changes by $\pm 1$.  If $L$ has a framing $\phi$, specified by a number $n$ and we give $L(g)$ the same framing $n$, then by \eqref{E:lift} the framing of $\ltilde(g)$ has changed by $\pm 1$.  More generally, we can do this for any finite collection of elements $g \neq 1$. Making use of \eqref{E:lift} to compute the framing upstairs, we get the following. 
\begin{lemma}\label{L:self-clasp}
For any null-homotopic framed knot $(L,\phi)$ with framing $n$ and any
$$
q = \sum_{1\neq g \in \pi} a_g g \in \Z[\pi]
$$
there is a null-homotopic framed knot $(L(q),\phi(q))$ with framing $n$ such that 
$$
\slktilde(\ltilde(q),\phi(q)) = \slktilde(\ltilde,\phi) + \sum_{1\neq g \in \pi} a_g
$$
\end{lemma}

\begin{figure}[htb]
\labellist
\small\hair 2pt
\pinlabel {$L_1(g)$} [ ] at 265 180
\pinlabel {$L_2$} [ ] at 410 180
\pinlabel {$L_1$} [ ] at 18 180
\pinlabel {$L_2$} [ ] at 162 180
\pinlabel {$g$} [ ] at 95 150
\pinlabel {$\Longrightarrow$} [ ] at 200 110
\pinlabel {$\beta_2$} [ ] at 115 60
\pinlabel {$\beta_1^{-1}$} [ ] at 30 59
\endlabellist
\centering
\includegraphics[scale=0.8]{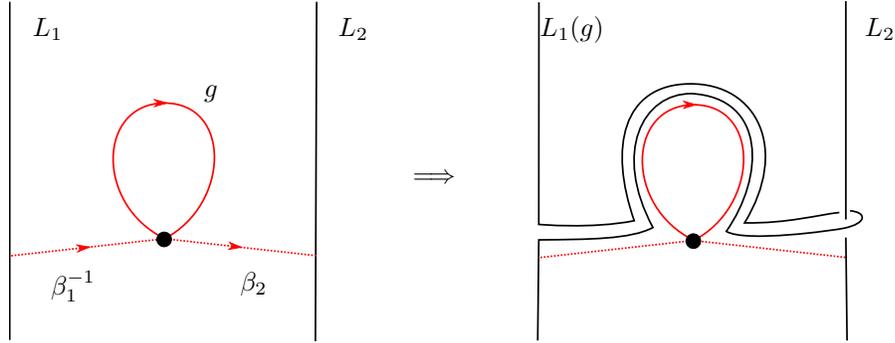}
\caption{Changing the twisted linking number}
\label{F:clasp}
\end{figure}

\subsection{Outline of the proof}
The main step in the proof is to construct a knot $k_p$ in $S^3$ so that $S^3_4(k_p)$ is topologically homology cobordant to $Q_m \conn  Y_p$. The homology cobordism will be glued onto $X_4(k_p)$ to give a topological $4$-manifold with a spine; we will argue that the resulting manifold is homeomorphic to $W_p$, and hence $W_p$ will have a topological spine.  There are several remarks before we embark on the proof. The first is that in the topological category, $Y_p$ is homology cobordant to $S^3$, so we will instead find a homology cobordism between $S^3_4(k_p)$ and $Q_m$. At the end, we will have to reassemble everything to get the correct boundary for the manifold we construct. 

A more substantial remark is that constructing a homology cobordism is an instance of {\em homology surgery}~\cite{cappell-shaneson:surgery}, since one is trying to create a manifold with a given homological rather than homotopy type. It is a standard observation, based on the Casson-Gordon invariants~\cite{casson-gordon:stanford,casson-gordon:orsay}, that homology surgery does not work well in dimension $4$, even in the topological category; see for example~\cite{akbulut:cg} or~\cite{levine-orr:survey}. So we will look to use something more like the usual surgery theory, which will produce a stronger conclusion.  In particular, we will look for a homology cobordism $V_m$ between $S^3_4(k_p)$ and $Q_m$ with fundamental group $G_m$ such that the induced cover of $S^3_4(k_p)$ is a homology sphere; in other words, $V_m$ will be a $\Z[G_m]$-homology cobordism.  This is in a sense similar to applications~\cite{friedl-teichner:slice} of surgery to knot slicing problems, where a homology surgery problem is solved by a judicious choice of fundamental group, so that Wall's surgery theory can be used directly. 

The construction of $V_m$ has the following steps.
\begin{enumerate}
\item Find a knot $K_p$  such that there is a map $f: S^3_4(K_p) \to Q_m$ that is a  $\Z[G_m]$-homology  equivalence. \label{kp}
\item Construct a normal cobordism between $f$ and the identity map from $Q_m$ to itself.\label{normal}
\item Modify $K_p$, producing $k_p$ and normal cobordism so that the surgery obstruction in $L(\Z[G_m])$ is trivial.\label{obstruction}
\end{enumerate}
The knots $K_p$ in this outline are not the same as the knots $K_p$ in~\cite{levine-lidman:spineless}, which play no role in the current work.

The second step turns out to be a straightforward deduction from known calculations, but the other two require a little more work. For instance, when $p=1$, one can find a knot $L$ (in fact the knot $12n0749$) whose knot group maps to $G_{-7}$ such that the induced $28$-fold cover of $S^3_4(L)$ is a homology sphere. However, there is no degree one map inducing this homology equivalence and in fact we couldn't find any knot in the tables (through 13 crossings, after which we gave up) with the desired properties for $K_1$. A general construction for $K_p$ is given in the next section.


\section{Degree one maps}\label{S:degree}
In this section we carry out the first step in the construction outlined above. Unless otherwise specified, we write $m=-4p-3$ for an arbitrary integer $p$. Recall~\cite{browder:surgery} that a normal map $f:M \to X$ from a manifold $M$ to a Poincar\'e complex $X$ is a degree one map with a choice of bundle map covering $f$ from the stable normal bundle (in the appropriate category) of $M$ to a lift of the Spivak normal fibration of $X$. We will deal with only the special case when both $X$ and $M$ are manifolds with trivial stable normal bundles, so any degree one map is covered (not uniquely) by such a bundle map and can be considered as a normal map.
\begin{proposition}\label{P:deg}
There is a knot $K_p$ and a degree one normal map $f: S^3_4(K_p) \to Q_m$ that is a  $\Z[G_m]$-homology  equivalence. 
\end{proposition}

We will find the following terminology to be useful.
\begin{definition}
Let $(J,K)$ be a link in $S^3$, and let $\vec\eta \subset S^3 - (J \cup K)$ be a link.
\begin{enumerate}
\renewcommand{\theenumi}{\alph{enumi}}
\item $\vec\eta$ is a {\em mod $K$ unlink} if its components bound disjoint discs in $S^3-J$.
\item $\vec\eta$ is a {\em mod $K$  unknotting link} if the linking number of every component $\eta_i$ with $J$ is $0$ and if doing a series of right or left twists along the components of $\vec\eta$ unknots $J$.
\end{enumerate}
\end{definition}
The mod $K$ terminology is intended to suggest that $\vec\eta$ has the desired property after we fill in $K$.

A very useful technique for constructing maps of degree one is the following~\cite[Proposition 3.2]{boileau-wang:bundles} (see also~\cite{gadgil:degree-one} for a converse result.)  Suppose $M$ is a 3-manifold, and $C$ a null homotopic curve. Then the result of Dehn surgery on $C$ with any slope maps to $M$ with degree one. The proof of this fact in~\cite{boileau-wang:bundles} is straightforward, but we offer the following $4$-dimensional version of the proof, which seems perhaps a bit easier. 
\begin{lemma}\label{L:degree}
Let $M$ be a $3$-manifold, and let $\vec C = \{C_1,\ldots,C_n\}$ be a framed link with all components null-homotopic. If $N$ is the result of surgery on $\vec C$ (with the given framings) then there is a degree one map from $N$ to $M$.
\end{lemma}
Note that this implies the stronger statement where surgery is replaced by Dehn surgery (or more properly, Dehn filling). The standard argument~\cite[Chapter 9.H]{rolfsen:knots} that a Dehn surgery can be done as a sequence of `honest' surgeries applies in an arbitrary $3$-manifold. The auxiliary curves that are added in this process are null homotopic, and so one can apply Lemma~\ref{L:degree}.
\begin{proof}
Let $X$ be the result of adding $2$-handles to $M \times I$ along $\vec{C} \times \{1\}$ with the given framings, so that $\partial X \cong N \cup M$. Since each component of $\vec{C}$ is null-homotopic, the identity map of $M$ extends to a retraction from $X$ to $M \times \{0\}$. Since $N$ is homologous to $M$ in $X$, the restriction of this map to $N$ has degree one.
\end{proof}

\subsection{Proof of Proposition~\ref{P:deg}}
We will assume $p \geq 0$ for simplicity; the case of negative $p$ is similar with slightly different pictures and notation.

To apply this consider the curves $\eta_1,\ldots,\eta_p$ in Figure~\ref{F:Qmeta2}(b). We would like to do $-1$-framed surgery along these curves to create a degree one map; the problem is that they are not trivial in $\pi_1(Q_m)$.
The curve $\eta_0$ will be used below in the proof of Lemma~\ref{L:K-unknot} to modify $\eta_1,\ldots,\eta_p$; the curve $z$ will be used afterwards to change the fundamental group.
\begin{figure}[htb]
\centering
\includegraphics[scale=.55]{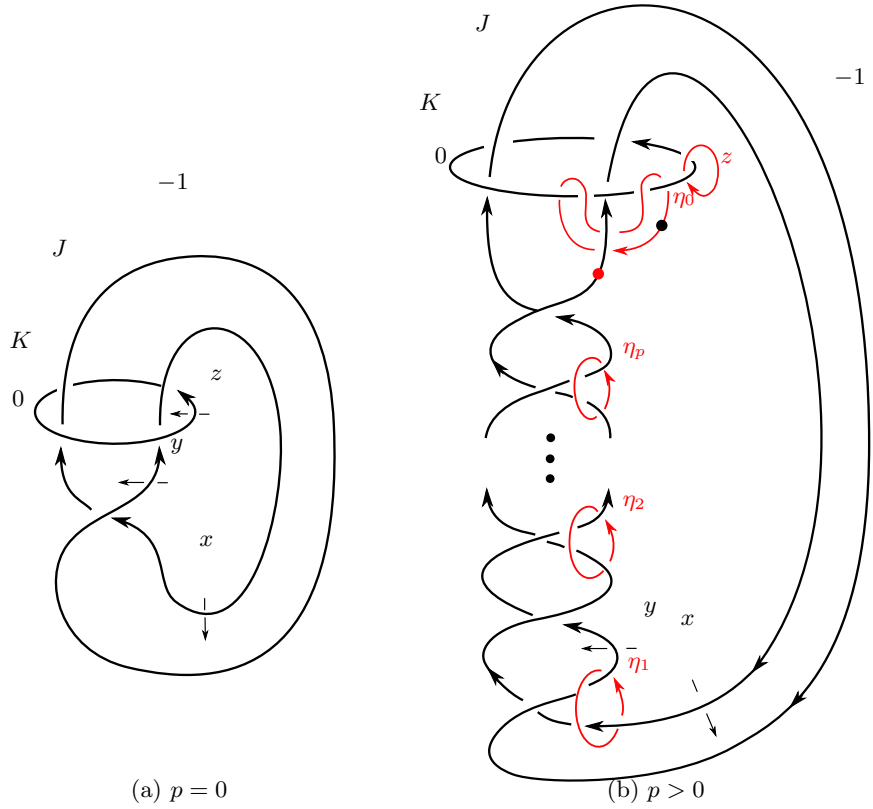}
\put(-150,285){\footnotesize$J$}
\put(-14,265){\footnotesize$-1$}
\put(-165,235){\footnotesize$0$}
\put(-171,255){\footnotesize$K$}
\put(-57,235){\footnotesize$\color{red}{z}$}
\put(-75,220){\footnotesize$\color{red}{\eta_0}$}
\put(-92,44){\footnotesize$\color{red}{\eta_1}$}
\put(-94,105){\footnotesize$\color{red}{\eta_2}$}
\put(-94,163){\footnotesize$\color{red}{\eta_p}$}
\put(-86,65){\footnotesize$y$}
\put(-72,60){\footnotesize$x$}
\put(-270,225){\footnotesize$-1$}
\put(-310,200){\footnotesize$J$}
\put(-325,143){\footnotesize$0$}
\put(-326,165){\footnotesize$K$}
\put(-250,153){\footnotesize$z$}
\put(-265,127){\footnotesize$y$}
\put(-254,90){\footnotesize$x$}
\put(-280,-05){\footnotesize(a) $p=0$}
\put(-100,-05){\footnotesize(b) $p>0$}
\caption{Surgery curves in $Q_{m}$}
\label{F:Qmeta2}
\end{figure}

Note that $\vec\eta = \{\eta_1,\ldots,\eta_p\}$ is a mod $K$ unknotting link for $J$ and that $\eta_0$ is an unknot mod $K$. It follows easily that there is a band sum of $\vec\eta$ with $p$ copies of $\eta_0$ to make a link $\vec\gamma = \{\gamma_1,\ldots,\gamma_p\}$  that is also a  mod $K$ unknotting link for $J$ and so that the homotopy class of $\gamma_i$ is $\eta_i^{-1}\ast \eta_0$. The calculation in the proof of Lemma~\ref{L:K-unknot} shows that $\eta_i^{-1}\ast \eta_0$ is trivial in $\pi_1(Q_{m})$ for all $i$.  

Note that we can find other mod $K$ unknotting links for $J$ by taking the band sum of components of $\vec\gamma$ with any mod $K$ unknot; we need to know which homotopy classes contain such knots. 
\begin{lemma}\label{L:K-unknot}
Every element of $\pi_1(Q_{m})$ is represented by a  mod $K$ unknot.
\end{lemma}  
\begin{proof}
We claim that $\pi_1(Q_{m})$ is generated by the curves $\eta_0$ (with the indicated base-point) and $z$. Both of these are mod $K$ unknots, so any product of them is represented by a mod $K$ unknot. Recalling that $m = -4p-3$, it is known that $Q_{m}$ can be represented by a surgery along $(-1)$-surgery on the $(2,2p+1) $ torus knot $J$ and $0$-surgery on a curve grabbing twice $K$ as in Figure~\ref{F:Qmeta2}(b) which is drawn for positive $p$. Let $g_p$ be the fundamental group of the complement of the link $(J,K)$ in $S^3$. For $p=0$, a Wirtinger presentation shows that $g_0$ is generated by the curves labeled $x, y, z$  in Figure~\ref{F:Qmeta2}(a) with relations $y^{-1}x(yzy^{-1}z^{-1})=1$, $yzx^{-1}z^{-1}=1$, and $yz^{-1}y^{-1}x^{-1}zx=1$. Note that the complement of the link $(J,K)$  for all the $p$ is the same as for $p=0$ since we do a twist around $K$, where the twist diffeomorphism from the complement of the link for $p=0$ to that for arbitrary $p$ takes $x\mapsto x$, $y\mapsto y$, and $z\mapsto (xy)^pz$. So, the fundamental group $g_p$ is generated by curves $x, y, z_p$ where $z_p$ denotes $(xy)^pz$, with relations $ y^{-1}x(yz_py^{-1}z_p^{-1})=1$, $yz_px^{-1}z_p^{-1} =1$, and $yz_p^{-1}y^{-1}x^{-1}z_px=1$.

The fundamental group $\pi_1(Q_{m})$ is obtained from $g_p$ simply by adding the relations induced from surgery curves $S_0$, $S_{-1}$, which are parallel to $K$ and $J$, with  framings $0$ and $-1$ respectively. From Figure~\ref{F:Qmeta2}, we can read off a presentation for each surgery curve; starting near the dot on curve of $J$ we get $1=zy^{p+1}zy^{-2p-2}y^{-p}=zy^{p+1}zy^{-3p-2}$  for $S_{-1}$ and obviously, the $0$-surgery $S_0$ gives $xy=1$.

So, we get a presentation for $\pi_1(Q_{m})$ with generators $x, y, z_p$ and relations 
\begin{align}
 y^{-1}x(yz_py^{-1}z_p^{-1}) & = 1,\\
yz_px^{-1}z_p^{-1} &= 1, \\
yz_p^{-1}y^{-1}x^{-1}z_px & = 1,\\
xy & = 1, \label{E:S0} \\
zy^{p+1}zy^{-3p-2}  & = 1 
\end{align}

Using \eqref{E:S0}, we get rid of the generator $x$ and the relation $(6)$ so that the presentation is of the form with generators $y, z$ and relations;
\begin{equation}\label{presentation2Qm}
y^{-2}(yzy^{-1}z^{-1})=1,\  \ yzyz^{-1}=1,\ \text{and} \  \ zy^{p+1}zy^{-3p-2}=1.
\end{equation}

The first two relations in~\eqref{presentation2Qm} are equivalent to $z^{-1}yz=y^{-1}$, which implies that $z^{-1}y^{p+1}z=y^{-p-1}$, so $y^{p+1}z=zy^{-p-1}$. Then the third relation $1=zy^{p+1}zy^{-3p-2}=z(zy^{-p-1})y^{-3p-2}=z^2y^{-4p-3}$. So, this reduction gives a presentation for $\pi_1(Q_{m})$ as follows;
\begin{equation}\label{fundgpQm}
 \la y, z \mid z^2=y^{4p+3}, z^{-1}yz=y^{-1} \ra .
\end{equation}
From Figure~\ref{F:Qmeta2}, we read off $\eta_0=yzy^{-1}z^{-1}$ and $\eta_1=y^2$, so the first relation in~\eqref{presentation2Qm} implies that $1=y^{-2}(yzy^{-1}z^{-1})=\eta_1^{-1}\ast \eta_0=y^{-2}\eta_0=1$, so $y^2=\eta_0$. This makes the relation in~\eqref{fundgpQm} 
 $z^2=y^{4p+3}=y^{2(2p+1)}\cdot y=\eta_0^{2p+1}\cdot y$, so $\eta_0^{-(2p+1)}z^2=y$, which shows that $\pi_1(Q_{m})$ is generated by $\eta_0$ and $z$.  This completes the proof of Lemma~\ref{L:K-unknot}.
\end{proof}

Returning to the proof of Proposition~\ref{P:deg}, note that  relation \eqref{E:S0} implies that the curves $\eta_1,\ldots,\eta_p$ all represent the class $y^2\in G_m = \pi_1(Q_{m})$, as does $\eta_0$.
Let $N_m$ be the result of $(-1)$ framed surgery on each component of  $\vec\gamma$. Then by the discussion above, there is a degree one map from $N_m$ to $Q_{m}$.  By construction, after surgery on $\vec\gamma$, we can blow down $\vec\gamma$ and $J$ in succession. The result is $+4$ surgery on some new knot $\hat{K}_m$ in the $3$-sphere, which can in principle be drawn. It will be a little messy, and is about to get a lot messier so we do not draw it.

At this point we have achieved most of step \eqref{kp}, since $N_m = S^3_4(\hat{K}_m)$ maps with degree one to $Q_{m}$, and hence its fundamental group maps surjectively onto $G_m$ and inducing an isomorphism on integral homology. However, there is no reason that it should induce an isomorphism in twisted homology. For example, a direct calculation using SnapPy~\cite{SnapPy} and GAP~\cite{GAP4} reveals that for $m=-7$, the homology of the induced $28$-fold cover $\ntilde_{-7}$ is quite large. We show how to modify $\vec\gamma$ (and hence $\hat{K}_m$) so as to kill the homology of $\ntilde_m$.

By construction, $\ntilde_m$ is surgery on the preimage of $\gamma$ in the universal cover of $Q_{m}$, which is the $3$-sphere.  As such, its homology is presented by the $p|G_m| \times p|G_m|$ matrix of linking numbers $\LL = \lk_{S^3}(\tilde{\gamma_i}\cdot g,\tilde{\gamma}_j' \cdot h)$ where $\gamma_j'$ is the $(-1)$-pushoff of $\gamma_j$. The framing on $\tilde{\gamma}_j$ is determined by Equation \eqref{E:lift}.  Using Lemma~\ref{L:self-clasp} we can modify $\gamma$ by self-clasping moves along elements of $\pi = G_{m}$ so that the first row of $\LL$ becomes the vector $(n',0,\ldots,0)$.  Since the framing on $\gamma$ was $-1$, Equation \eqref{E:lift} says that $n'= -1$ as well. 
 
By Lemma~\ref{L:K-unknot}, every element of $G_{m}$ is represented by a  mod $K$ unknot, so we can do the self-clasping moves so that modified version of $\vec\gamma$ remains a mod $K$ unknotting link for $J$. It follows that the result of blowing down the modified $\vec\gamma$, followed by blowing down $J$, turns $K$ into a knot $K_p$ with all the desired properties. 

Finally, since $S^3_4(K_p)$ and $Q_m$ are both parallelizable, any degree one map is covered by a map of stable normal bundles, and hence is a normal map. \qed
\begin{remark}
This argument for altering the homology of a covering space was inspired by Jerry Levine's surgical construction of a knot with a given Alexander polynomial~\cite{levine:characterization}.
\end{remark}

\section{The surgery problem}\label{S:surgery}
First we show that the map constructed in the previous section gives rise to a normal map to $Q_m \times I$.
\begin{proposition}\label{P:normal}
Any degree one normal map $f:Y\to Q_m$ that is a  $\Z[G_m]$-homology  equivalence is normally cobordant to the identity map of $Q_m$.
\end{proposition}

\begin{proof}
We make use of a version~\cite{kirby-taylor:surgery} of the surgery exact sequence for $3$-manifolds, in which the usual structure set $\Sm(M)$ for homotopy equivalences to a manifold $M^3$ is replaced by the set of $\Z[\pi_1(M)]$-homology equivalences.  We are in a setting where we don't know (or care) if the homology equivalence is simple. Hence the relevant surgery groups are $L^h$ groups, and the equivalence relation for the structure set is up to $h$-cobordism instead of homeomorphism or diffeomorphism.  For our purposes, it doesn't matter whether we are working in the topological or smooth category; we will stick to $\TOP$ for convenience.  Writing for $\bar{\Sm}(M)$ for this modified structure set, there is an exact sequence
\begin{equation}\label{E:surgery}
\begin{CD}
\bar{\Sm}(M)  @>N>>  [M,G/\TOP]  @>\theta>> L_3^h(\Z[\pi_1(M)])
\end{CD}
\end{equation}
The first map is the normal invariant, while the second is the surgery obstruction.  The maps in this sequence have been studied in some detail in~\cite{jahren-kwasik:three}, and we could base a proof of Proposition~\ref{P:normal} on Theorem 2 of that paper. However, we give an alternate and perhaps slightly more direct argument.

By the well-known calculation~\cite{kirby-siebenmann:triangulation} of the low-dimensional homotopy groups of $G/\TOP$, we get that 
$$
[Q_m,G/\TOP]  = [Q_m, K(\Z_2,1)] = H^1(Q_m;\Z_2) = \Z_2.
$$
The map to $\Z_2$ is a codimension-one Arf invariant, and we claim that it vanishes.  Because the projection $G_m \to \Z_4$ splits ($G_m$ being a semi-direct product as in \eqref{E:Gm}) the induced map $L_3(\Z[G_m]) \to L_3(\Z[\Z_4]) $  is a split surjection.  It is known~\cite{hambleton-taylor:guide} that $L_3(\Z[\Z_4]) \cong \Z_2$, with the non-trivial element being detected by a codimension-one Arf invariant. So in this case, the surgery obstruction map $\theta$ is an injection.  Since the composition $\theta \circ N$ is trivial, and our given $f$ is a $\Z[G_m]$-homology equivalence, its normal invariant must be trivial. 
\end{proof}

\subsection{Killing the surgery obstruction}\label{S:obstruction}
At this point, we have a normal map $F: Z^4 \to Q_m \times I$ where $\partial Z = S^3_4(K_p) \coprod -Q_m$. The restriction of $F$ to $Q_m$ is the identity, and the restriction to $S^3_4(K_p)$ is a $\Z[G_m]$-homology equivalence.  Hence there is a well-defined surgery obstruction $\theta(F) \in L_4(\Z[G_m])$ to doing surgery on $Z$ to obtain a homotopy $Q_m \times I$. In principle, the groups $L_4(\Z[G_m])$ are computable~\cite{wall:VI,hambleton-taylor:guide} but in practice there are $2$-torsion obstructions that are hard to evaluate. 

Our approach is to alter the knot $K_p$ one more time to kill the surgery obstruction. The main step is worth summarizing in a lemma; it applies both to $L_4^s$ and $L_4^h$, either of which will be temporarily denoted $L_4$.  Recall~\cite[Chapter 5]{wall:book} that an element $\theta$ of $L_4(\Z[G])$ may be represented, possibly after stabilizations, by a free module $A$ over $\Z[G]$ and a pair $(\lambda,\mu)$. Here $\lambda$ is a form on $A\times A$ that is Hermitian with respect to the involution induced by $g \to g^{-1}$ on $\Z[G]$ and $\mu$ is a quadratic form on $A$ related to $\lambda$.  Let $B$ be a matrix for $\lambda$; the image of $B$ under the augmentation $\epsilon: \Z[G] \to \Z$ has a signature which we denote by $\sign(\theta)$.
\begin{lemma}\label{L:obstruction}
Suppose that $K$ is a knot in $S^3$, and that there is a $3$-manifold $Y$ along with a $\Z[G]$-homology equivalence $f:S^3_n(K) \to Y$ (here $G= \pi_1(Y)$). Then there is a knot $K'$ and a topological normal cobordism  $(Z,F)$ from $f$ to a $\Z[G]$-homology equivalence $f':S^3_n(K') \to Y$ with surgery obstruction $\theta(Z,F) = \theta$. If $\sign(\theta) \equiv 0 \pmod{16}$ then there is a smooth normal cobordism.
\end{lemma}
\begin{proof}
We remark first that any finite collection of elements of $\pi_1(S^3_n(K))$ may be represented by an unlink. For one can start with any link $L$ in $S^3- K$ representing those group elements, and then do crossing changes in $S^3$, missing $K$, to make $L$ into an unlink in $S^3-K$.

Note that $\epsilon(\theta)$ is an even unimodular form over $\Z$. This follows from~\cite[Theorem 5.2 (iii)]{wall:book}, which depends on the fact that $F$ is a normal map. A more direct argument is that since $F$ is a normal map, $Z$ is in fact a spin manifold, and hence its intersection form $\epsilon(\theta)$ is even.

By the previous remarks the signature of $\epsilon(\theta)$ is divisible by $8$.   Since there exist closed smooth spin manifolds of signature $16$ and closed topological spin manifolds of index $8$, we can change $\sign(\theta)$ by multiples of $8$ or $16$ in the topological (resp. smooth) categories without affecting the boundary.  So in either category, our assumption on the signature means that we may assume that $\sign(\theta)=0$. Since $L_4(\Z)$ is detected by the signature, this means that $\epsilon(\theta)$ is trivial in $L_4(\Z)$.

As described above, $\lambda$ is encoded as a matrix with respect to a basis for a free $\Z[G]$ module $A$.  Note that any integral change of basis for 
$$
A \otimes_{\Z[G]} \Z
$$
lifts to a change of basis for $A$. Since $\epsilon(\theta) = 0 \in L_4(\Z)$, it follows that there is a basis  $\{e_i,\,|\, i=1\ldots 2m\}$ such that $\epsilon(\lambda)$ is represented by a sum of hyperbolic forms
$$
H = 
\begin{pmatrix}
0&1\\1&0
\end{pmatrix}.
$$
Let $B$ be the matrix for $\lambda$ with respect to this basis.

Now we follow the proof of Wall's realization theorem~\cite[Theorem 5.8]{wall:book}, adding handles to $S^3_n(K) \times I$ along a framed link in $S^3_n(K) \times \{1\}$ to create the normal cobordism $Z$. More precisely, start with a $0$-framed unlink $\vec{C}$ in a ball $S^3_n(K)$ whose components $C_i$ correspond to the basis vectors $e_{i}$. As in Lemma~\ref{L:clasp}, for $i\neq j$ we introduce clasps between the circles $C_i$ and $C_j$ so that the twisted linking number (measured in $\Z[G]$) between those circles becomes $B_{ij}$.  Similarly, introduce self-clasps into $C_i$ along elements of $\pi_1(S^3_n(K))$ so that the self-intersection form will be given by $\mu$.  When handles are added along the components of the resulting $0$-framed link, the intersection form of the resulting $4$-manifold will be given by $\lambda$. 

It follows that the upper boundary of $Z$, which results from doing surgery along $\vec{C}$, is $n$-framed surgery on the knot $K'$ obtained by cancelling all of the handles (viewed as attached to $C$ in $S^3$).
\end{proof}
\begin{remark}
In the case of interest, we let $Y=Q_m$, and apply Proposition~\ref{P:deg} to obtain the knot $K$, and then Lemma~\ref{L:obstruction} to change it into $K'$.  In principle, one could trace through their proofs to determine the knots $K$ and $K'$, respectively. In the special case $p=1$, the group $L_4(\Z[G_{-7}]) \cong \Z^6$, detected by multisignatures, which are computable. So perhaps in this case, the knot $K'$ could be drawn. 
\end{remark}
\section{Proof of Theorem~\ref{thm:spine}}
Putting together Propositions~\ref{P:deg} and \ref{P:normal}, we find a knot $K_p$, a $\Z[G_m]$-homology equivalence $f: S^3_4(K_p) \to Q_m$, and a normal cobordism $(Z,F)$ from $f$ to the identity of $Q_m$.  Making use of Lemma~\ref{L:obstruction}, we find another knot $k_p$, and a normal cobordism $(Z',F')$ (possibly non-smoothable) from $f$ to a $\Z[G_m]$-homology equivalence $f': S^3_4(k_p) \to Q_m$ such that the surgery obstruction $\theta(F')$ is the negative of $\theta(F)$.  Gluing these two together gives a normal cobordism from $f'$ to  the identity of $Q_m$ with trivial surgery obstruction.

Since $\pi_1(Q_m)$ is finite, Freedman's theorem~\cite{freedman-quinn} says that surgery on 
$$
Z \cup_{S^3_4(K_p)} Z'
$$
can be done to get a homotopy equivalence from a new $4$-manifold $V_p$ to $Q_m \times I$.  

Now recall that Levine-Lidman's manifold is denoted $W_p$, and its boundary is $Q_m \conn  -Y_p	$ where $Y_p$ is a homology sphere.  In the topological category, $Y_p$ bounds a contractible manifold $U$, and we take the boundary connected sum of $V_p$ with $U$ manifold to get a homology cobordism between $S_4^3(k_p)$ and $\partial W_p$.  By the $d$-invariant calculation in~\cite{levine-lidman:spineless} this homology cobordism cannot be smooth, unless $p \in\{0,-1,-2\}$.
Let
$$
W'_p  = X_4(k_p) \cup_{S_4^3(k_p)} (V_p \mathop\natural U)
$$
and note that $\partial W'_p=Q_m \conn  -Y_p	$.

By construction, $W'_p$ has a topological spine, so Theorem~\ref{thm:spine} will be proved if we can identify it with $W_p$.\\[1ex]
{\bf Claim:} $W'_p$ is homeomorphic to $W_p$.\\[1ex]
To prove the claim, we use the work of Boyer~\cite{boyer:boundary}, which extends Freedman's classification of simply-connected topological manifolds to the setting of manifolds with boundary. 

By construction, the inclusion of $S_4^3(k_p)$ into $V_p$ induces a surjection on the fundamental group. Since $X_4(k_p)$ is simply connected, it follows that $W'_p$ is simply connected as well.  Note further that $W'_p$ is spin, since its intersection form is even. The  classification for general boundaries is rather complicated, but in the present case the answer turns out to be quite simple. For an oriented $3$-manifold $M$, Boyer constructs a map $c^t_L(M)$ from the set of homeomorphism classes of simply connected  $4$-manifolds with boundary $M$ and intersection form $L$ to a certain coset space $B^t_L(M)$ of the automorphisms of the linking form of $M$.  When $L$ is even, and $M$ is a rational homology sphere, this map is an injection.  

The fact that $H_1(Q_m \conn  -Y_p) \cong \Z_4$ implies~\cite[Proposition 0.6]{boyer:boundary} that $B^t_L(Q_m \conn  -Y_p)$ has exactly one element. Hence $W'_p$ is homeomorphic to $W_p$ and the proof of Theorem~\ref{thm:spine} is complete. \qed
\begin{remark}
It is reasonable to wonder if the other manifolds $W_{k,m}$ constructed in~\cite{levine-lidman:spineless} have topological spines. These manifolds have intersection form $\langle k^2\rangle$ for $k>2$ and $|m| \gg 0$, and boundaries of the form $Q_{k,m}\conn  Y$. Here $Q_{k,m}$ is surgery on a link similar to the one in Figure~\ref{F:Qmeta2}(b) and $Y$ is a homology sphere. So it seems likely that the portion of the argument that produces a knot and a degree homology equivalence on the result of surgery on that knot to $Q_{k,m}$ will still work. However, for our method to succeed, we need an honest surgery problem whose restriction to the boundary is a homology equivalence with coefficients in $\pi_1(Q_{k,m})$, and not just a homology surgery problem. The fundamental group of $Q_{k,m}$ is large and it is not clear what should be the fundamental group of a surgery problem that might be constructed to produce a homology cobordism.   
\end{remark}

\end{document}